\documentclass[english,pdftex,twoside]{smfart}
\usepackage[colorlinks=true,linkcolor=blue,citecolor=blue,urlcolor=blue]{hyperref}
\usepackage{amsmath}
\usepackage{amssymb}
\usepackage{mathrsfs}
\usepackage[all]{xy}
\usepackage{graphicx}
\usepackage{latexsym}
\usepackage{amsthm}
\usepackage{amscd}
\usepackage{mathrsfs}

\RequirePackage{fancyhdr}
\RequirePackage{times}
\RequirePackage{acronym}

\fancypagestyle{plain}{%
}

\newtheorem{proposition}{Proposition}
\newtheorem{theorem}[proposition]{Theorem}
\newtheorem{lemma}[proposition]{Lemma}
\newtheorem{corollary}[proposition]{Corollary}
\theoremstyle{definition}

\newtheorem*{remark}{Remark}
\newtheorem*{example}{Example}

\def\CC{{\Bbb C}}

\title{Parametric solutions of Pell equations}
\author{Leonardo Zapponi}
\date{\today}
\email{leonardo.zapponi@imj-prg.fr}

\begin{document}

\maketitle

\begin{abstract} This short paper is concerned with polynomial Pell equations
\[P^2-DQ^2=1,\]
with $P,Q,D\in\Bbb C[X]$ and $\deg(D)=2$. The main result shows that the polynomials $P$ and $Q$ are closely related to Chebyshev polynomials. We then investigate the existence of such polynomials in $\Bbb Z[X]$ specializing to fixed solutions of ordinary Pell equations over the integers.
\end{abstract}

\section*{Introduction}

This paper is a digestet, condensed and updated version of a discussion recently held in the mathematical forum Mathoverflow~\cite{K}. The original and motivating question, asked by Stefan Kohl, can be summarized as follows:

\vskip.4cm

\begin{center}\begin{minipage}[c]{.92\textwidth}
{\it\hskip.4cm Let $n$ be a positive integer which is not a perfect square and consider a fundamental solution $(a,b)$ of the Pell equation
\[x^2-ny^2=1.\]
Is it always possible to construct polynomials $P,Q,D\in\Bbb Z[X]$, with $\deg(D)=2$, and an integer $k$, such that
\[P^2-DQ^2=1,\]
with $P(k)=a, Q(k)=b$ and $D(k)=n$? In case of positive answer, is it possible to bound the degree of $P$?}
\end{minipage}
\end{center}

\vskip.4cm

Numerical evidences suggested that given a fundamental solution of the Pell equation, the polynomials $P,Q$ and $D$ can always be constructed and that the degree of $P$ is at most $6$. David Speyer posted an explicit family with $P$ of degree $1$ and the author did the same for $P$ of degree $2$, leading to a positive answer to the first question. In a second post, the author was able to prove that, if the integer $n$ is square-free and not congruent to $1$ modulo $4$ then the degree of $P$ is at most $2$. The general case which was proved subsequently and just announced on the forum confirms the numerical evidences as soon as we suppose that $n$ is square-free.

The paper is organized as follows: in the first section, we study the general polynomial Pell equation
\[P^2-DQ^2=1,\]
with $D$ of degree $2$. The key result of the paper is theorem~\ref{th1}, which states that $P,Q$ and $D$ can be explicitly computed and that they are closely related to Chebyshev polynomials. One of the main ingredients of the proof comes from Grothendieck's theory of dessins d'enfants. The first two paragraphs of the second section describe all the families of polynomials $P,Q$ and $D$ satisfying the conditions in Stefan Kohl's question, provided that the degree of $P$ is less than or equal to $2$. More precisely, for $P$ of degree $1$, proposition~\ref{prop1} provides essentially the same family obtained by David Speyer, while for $P$ of degree $2$, proposition~\ref{prop2} gives a refined version of the family posted by the author. The last paragraph is concerned with the possible degrees of $P$. The main result is theorem~\ref{th2} which asserts that for a fixed solution $(a,b)$, the degree is bounded by an explicit constant depending on $(a,b)$ and corollary~\ref{cor1} establishes that without any restriction on the integer $n$, there is no uniform bound. Nevertheless, if $n$ is square-free, corollary~\ref{cor2} claims that the degree is bounded by $6$.

\section*{Aknowledgments}

The author would like to thank Stefan Kohl for asking the question which inspired this work and for the many comments and suggestions posted on the forum. A special thank goes to David Speyer, reading his post and the following comments leaded to a better understanding of the problem. Finally, the amount of anonimous readers consulting the post (more than a thousand, counted with multiplicity) has been a motivating implulse for the redaction of this note; a last thank is adressed to them.

\section{Chebyshev polynomials}

Let $K=\Bbb C(X)$ denote the field of rational functions over the complex numbers. The set of sequences $(u_n)$ of elements of $K$ satisfying the relation
\[u_{n+1}=2Xu_n-u_{n-1}\]
is a $K$-vector space of dimension $2$, spanned by the sequences $(T_n)$ and $(U_n)$ defined by
\[\begin{cases}
T_0=1,\\
T_1=X,
\end{cases}\qquad\mbox{and}\qquad\begin{cases}
U_0=1,\\
U_1=2X.
\end{cases}\]

For any non-negative integer $n$, the polynomials $T_n,U_n\in\Bbb Z[X]$ are respectively called {\it Chebyshev polynomials of the first and second kind of degree} $n$. It is easily checked that in the field $L=K(\sqrt{X^2-1})$, they satisfy the identity
\[\left(X+\sqrt{X^2-1}\right)^n=T_n+U_{n-1}\sqrt{X^2-1},\]
which can be considered as an alternative definition. In particular, the {\it Pell relation}
\begin{equation}\label{eq1}
T_n^2-(X^2-1)U_{n-1}^2=1
\end{equation}
holds. This last identity will be the main ingredient in the following sections.

From a geometric point of view, the polynomial $T_n$ induces a cover of the projective line unramified outside $\infty,1$ and $-1$. It follows that the morphism $f_n:\Bbb P^1\to\Bbb P^1$ associated to the polynomial
\[f=\frac12(T_n+1)\]
is a {\it Belyi map}. Its ramification data of depends on the parity od $n$. In any case, it is totally ramified above $\infty$. If $n$ is even then all the points above $1$ have ramification index $2$ and all the points above $1$ have ramification index $2$ expcepted two of them, which are unramified. If $n$ is odd then the ramification behaviour above $0$ and $1$ is the same: all the points have ramification index $2$ excepted one of them, which is unramified. Grothendieck's theory of {\it dessins d'enfants} provides an elegant combinatorial description of the isomorphism classes of Belyi maps. We will not enter into the details, but it follows that given a fixed positive integer $n$, there is a unique dessin d'enfant corresponding to the above ramification data (see the picture below). We refer to the book~\cite{D} for a complete introduction to the subject.

\vskip.4cm
\begin{figure}[htbp]
\begin{center}
\includegraphics[scale=.3]{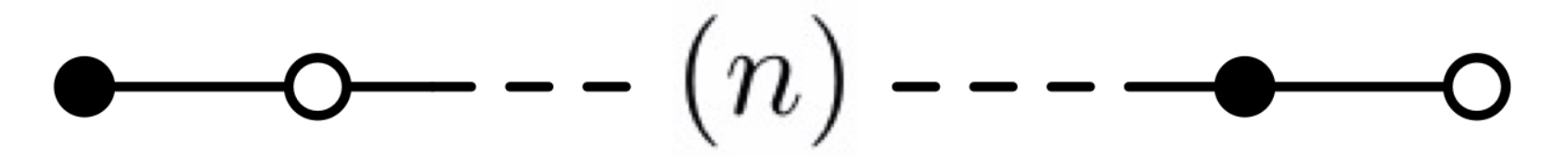}
\caption{The dessin d'enfant associated to $T_n$.}
\end{center}
\end{figure}

\section{Polynomial Pell equations}

We now describe the solutions of Pell equation
\[P^2-DQ^2=1,\]
with $P,Q,D\in\Bbb C[X]$, in the special case where $D$ has degree $2$. The following result shows that the polynomials $P$ and $Q$ are closely related to Chebyshev polynomials.

\begin{theorem}\label{th1} Let $P,Q,D\in\CC[X]$ with $\deg(D)=2$ and $\deg(P)=d$. The following conditions are equivalent:
\begin{enumerate}
\item We have the identity
\[P^2-DQ^2=1.\]
\item There exist constants $\lambda,\mu\in\Bbb C^\times$ and $\nu\in\Bbb C$ such that
\[\begin{cases}
P=\pm T_d(\lambda X+\nu),\\
Q=\mu U_{d-1}(\lambda X+\nu),\\
D=\mu^{-2}\left((\lambda X+\nu)^2-1\right).
\end{cases}\]
\end{enumerate}
\end{theorem}

\begin{proof} First of all, the relation~\ref{eq1} shows that the second condition implies the first. Suppose now that $P,Q,D\in\Bbb C[X]$ satisfy condition 1, with $\deg(D)=2$. Assuming $P\neq0$, it follows that $\deg(P)=d\geq1$ and thus $\deg(Q)=d-1$. Consider the polynomial $f=P^2$, so that $f'$ has degree $2d-1$. Setting
\[P=u\prod_{i=1}^r(X-x_i)^{e_i}\quad\mbox{and}\quad Q=v\prod_{i=1}^s(X-y_i)^{f_i},\]
with $u,v\in\Bbb C,r\leq d$ and $s\leq d-1$, we obtain the factorization
\[f'=\prod_{i=1}^r(X-x_i)^{2e_i-1}\prod_{i=1}^r(X-y_i)^{2f_i-1}R,\]
with $R\in\Bbb C[X]$. Since $d=\sum_{i=1}^re_i=1+\sum_{i=1}^sf_i$, we find the identity
\[2d-1=\sum_{i=1}^r(2e_i-1)+\sum_{i=1}^s(2f_i-1)+\deg(R)=4d-2-r-s+\deg(R),\]
which leads to
\[r+s=2d-1+\deg(R).\]
It then follows that $r=d,s=d-1$ and $\deg(R)=0$, i.e. $P$ and $Q$ are separable. Remark that the polynomial $D$ is then itself separable. In this case, the cover $\Bbb P^1\to\Bbb P^1$ induced by $f$ is only ramified above $\infty,0$ and $1$, i.e. it is a Belyi map. Moreover, it is totally ramified above $\infty$, all the points above $0$ have ramification index $2$ and the points above $1$ have ramification $2$ excepted two of them, which are unramified, corresponding to the roots of $D$. It then follows from the discussion in the first section that there exist constants $\lambda\in\Bbb C^\times$ and $\nu\in\Bbb C$, such that
\[f=\frac12\left(T_{2d}(\lambda X+\nu)+1\right).\]
Now, the general identity
\[T_{2d}=2T_d^2-1\]
leads to the relation
\[P^2=T_d^2(\lambda X+\nu),\]
from which we deduce the expression of $P$ in the theorem. Finally, the polynomials $Q$ and $D$ are obtained from relation~\ref{eq1} combined with the fact that the polynomial $U_{d-1}$ is separable.
\end{proof}

\begin{remark} If $d$ is odd, we can remove the $\pm$ sign in the expression of $P$ since in this case the Chebyshev polynomial $T_d$ defines an odd function, so that $T_d(-\lambda X-\nu)=-T_d(\lambda X+\nu)$.
\end{remark}

\section{Parametric solutions}

We are now ready to study Stefan Kohl's problem stated in the introduction, from which we keep the notation. We will refer to the polynomials $P,Q,R\in\Bbb C[X]$ as a  {\it parametric solution} associated to $(a,b)$, its {\it degree} being the integer $\deg(P)$. Remark that we can always reduce to the case $k=0$. From now on, we restrict to this situation. With this assumption, if $P,Q,D\in\Bbb Z[X]$ is a parametric solution, the same holds for the polynomials $P(mX),Q(mX)$ and $D(mX)$ for any integer $m\neq 0$. In the following paragraphs, we explicitly answer the first question by performing all the parametric solutions of degree $1$ or $2$ having inregral coefficients. We then bound the degree of general parametric solutions, showing, among the others, that if the integer $n$ is square-free then the degree is at most $6$.

\subsection{Degree 1}

We start by giving an explicit description of all parametric solutions with $P$ of degree $1$.

\begin{proposition}\label{prop1} Let $n$ be a positive integer which is not a perfect square and consider a solution $(a,b)$ of the Pell equation
\[x^2-ny^2=1.\]
Given three polynomials $P,Q,D\in\Bbb Z[X]$ with $\deg(P)=1$ and $\deg(D)=2$, the following conditions are equivalent:
\begin{enumerate}
\item We have the identity
\[P^2-DQ^2=1,\]
with $P(0)=a, Q(0)=b$ and $D(0)=n$.
\item Set $c=1$ if $b$ is odd and $c=2$ if $b$ is even. There exists an integer $m\neq0$ such that $P=P_0(c^{-1}mX), Q=b$ and $D=D_0(c^{-1}mX)$, with
\[\begin{cases}
P_0=b^2X+a,\\
D_0=b^2X^2+2aX+n.\\
\end{cases}\]
\end{enumerate}
\end{proposition}

\begin{proof} First of all, if the polynomials $P,Q$ and $D$ are defined as in condition 2, it is easily checked that they satisfy condition 1. Suppose now that $P,Q,D\in\Bbb Z[X]$ satisfy condition 1. The polynomial $Q$ is constant, so that $Q=b$. Following the notation in theorem 1, we obtain the identity $\mu=b$ and, taking account of the remark after theorem~\ref{th1}, we have the expression
\[P=\lambda X+\nu,\]
which implies that $\lambda$ is an integer and that $\nu=a$. Since $b^{-2}(a^2-1)=n$, we have the identities
\[D=b^{-2}\left((\lambda X+\nu)^2-1\right)=\lambda^2b^{-2}X^2+2\lambda ab^{-2} X+n,\]
and thus $b$ divides $\lambda$ and $b^2$ divides $2\lambda a$. The integers $a$ and $b$ being coprime, we deduce that $b^2$ divides $2\lambda$. Setting $\lambda=bu$, it follows that $b$ divides $cu$, say $u=c^{-1}bm$, which leads to the desired expressions.
\end{proof}

\subsection{Degree 2}

We now give a complete description of parametric solutions when the polynomial $P$ has degree $2$.

\begin{proposition}\label{prop2} Let $n$ be a positive integer which is not a perfect square and consider a solution $(a,b)$ of the Pell equation
\[x^2-ny^2=1.\]
Given three polynomials $P,Q,D\in\Bbb Z[X]$ with $\deg(P)=\deg(D)=2$, the following conditions are equivalent:
\begin{enumerate}
\item We have the identity
\[P^2-DQ^2=1,\]
with $P(0)=a, Q(0)=b$ and $D(0)=n$.
\item There exist two integers $m\neq0$ and $\varepsilon\in\{\pm1\}$ such that, setting
\[c=\mbox{gcd}\left(b^3,(a+\varepsilon)b,2(a+\varepsilon)^2\right),\]
we have $P=P_0(c^{-1}mX), Q=Q_0(c^{-1}mX)$ and $D=D_0(c^{-1}mX)$, with
\[\begin{cases}
P_0=b^4(a+\varepsilon)X^2+2b^2(a+\varepsilon)X+a,\\
Q_0=b^3X+b,\\
D_0=(a+\varepsilon)^2b^2X^2+2(a+\varepsilon)^2X+n.
\end{cases}\]
\end{enumerate}

\end{proposition}

\begin{proof} First of all, a direct computation shows that if $P,Q$ and $D$ are defined as in condition 2, they satisfy condition 1. Suppose now that $P,Q$ and $D$ satisfy the first condition. Following theorem~\ref{th1}, we have the expressions
\[\begin{cases}
\varepsilon P=2\lambda^2X^2+4\lambda\nu X+2\nu^2-1,\\
Q=2\lambda\mu X+2\nu\mu,\\
D=\lambda^2\mu^{-2}X^2+2\lambda\nu\mu^{-2}X+(\nu^2-1)\mu^{-2}.
\end{cases}\]
with $\varepsilon=\pm1$. The relations $P(0)=a$ and $Q(0)=b$ lead to the identities
\[\nu^2=\frac{\varepsilon a+1}2\qquad\mbox{and}\qquad\mu=\frac b{2\nu}.\]
In particular, we obtain the relation
\[(\nu^2-1)\mu^{-2}=n.\]
The leading coefficient of $Q$ being an integer, there exists $u\in\Bbb Z$ such that
\[\lambda=\frac{\nu u}b.\]
We then deduce that $D$ belongs to $\Bbb Z[X]$ if and only if $b^2$ divides $m(\varepsilon a+1)$ and $b^3$ divides $2u(\varepsilon a+1)^2$. These last two conditions can be summarized as the divisibility of\linebreak $u(\varepsilon a+1)\gcd(b,2(\varepsilon a+1))$ by $b^3$, which is itself equivalent to the divisibility of $u$ by $c^{-1}b^3$. Setting $u=c^{-1}b^3m$, we obtain the expressions in condition $2$.
\end{proof}

\begin{example} Setting $m=\pm1$ in the above proposition leads to parametric solutions whose coefficients have the smallest absolute value. For example, for $n=31,a=1520$ and $b=273$, setting $m=\varepsilon=1$, we find $c=59319$ and obtain the polynomials
\[\begin{cases}
P=2401X^2+3822X+1520,\\
Q=343X+273,\\
D=49X^2+78X+31,
\end{cases}\]
whose coefficients are much smaller than those of $P_0,Q_0$ and $D_0$, since we have the expressions
\[\begin{cases}
P_0=8448503770161X^2+226717218X+1520,\\
Q_0=20346417X+273,\\
D_0=172418444289X^2+4626882X+31.
\end{cases}\]
\end{example}

\subsection{Bounding the degree} We finally study the possible degrees of a parametric solution. We start with an easy but usefull lemma.

\begin{lemma}\label{lemm1} Let $P,Q,D\in\Bbb Q[X]$ be a paramteric solution of degree $d$ associated to $(a,b)$. Then there exist $P_1,Q_1,D_1\in\Bbb Z[X]$ defining a parametric solution of degree $d$ associated to $(a,b)$.
\end{lemma}

\begin{proof} By assumption, the constant terms of $P,Q$ and $D$ are integers. If $m$ denotes the lowest common multiple of the denominators of the coefficients of $P,Q$ and $D$ then the polynomials $P_1=P(mX),Q_1=Q(mX)$ and $D_1=D(mX)$ belong to $\Bbb Z[X]$ and define a parametric solution of degree $d$ associated to $(a,b)$.
\end{proof}

Let now $n$ be a positive integer which is not a square and consider the real quadratic field $K=\Bbb Q(\sqrt{n})$. Denote by $\mathcal O_K$ the ring of integers of $K$. The subgroup $U$ of $\mathcal O_K^\times$ consisting of elements having norm $1$ is isomorphic to $\Bbb Z\times\Bbb Z/2\Bbb Z$. The elements of the group $V=U\cap\Bbb Z[\sqrt{n}]$ bijectively correspond to the solutions of the Pell equation
\[x^2-ny^2=1.\]
Moreover, $U$ and $V$ are isomorphic to $\Bbb Z\times\Bbb Z/2\Bbb Z$, the quotient $U/V$ is cyclic and $V$ is generated by $-1$ and $a+b\sqrt{n}$, where $(a,b)$ is a fundamental solution. More generally, if $(a,b)$ is any solution of the above Pell equation, denote by $V(a,b)$ the subgroup of $V$ generated by $a+b\sqrt{n}$ and $-1$.

\begin{theorem}\label{th2}Let $n$ be a positive integer which is not a square and consider a solution $(a,b)$ of the Pell equation
\[x^2-ny^2=1.\]
The following conditions are equivalent:
\begin{enumerate}
\item There exists a parametric solution $P,Q,D\in\Bbb Z[X]$ of degree $d$ associated to $(a,b)$.
\item The integer $d$ divides $2(U:V(a,b))$.
\end{enumerate}
\end{theorem}

\begin{proof} We start by supposing that the first condition is fulfilled. Let $P,Q,D\in\Bbb Z[X]$ be a parametric solution of degree $d$. We know from theorem~\ref{th1} that there exist constants $\lambda,\mu\in\Bbb C^\times$ and $\nu\in\Bbb C$ such that
\[\begin{cases}
P=\pm T_d(\lambda X+\nu),\\
Q=\mu U_{d-1}(\lambda X+\nu),\\
D=\mu^{-2}\left((\lambda X+\nu)^2-1\right).
\end{cases}\]
We can restrict to the case where $P=T_d(\lambda X+\nu)$. Indeed, for $P=-T_d(\lambda X+\nu)$, if we replace $P$ with $-P$, we obtain $P(0)=-a$ and $Q(0)=b$, which still defines a solution of the same Pell equation and the groups $V(a,b)$ and $V(-a,b)$, which are isomorphic, have the same index in $U$. Considering the quotient of the coefficient of $X$ in $D$ by its leading coefficient, it follows that $\lambda\nu^{-1}$ is rational. Similarly, comparing the leading coefficients of $P$ and $Q$, we deduce that $\lambda\mu^{-1}$ is also rational. From the expression of the coefficient of $X$ in $D$, the same holds for $\nu\mu^{-1}$. Finally, taking the quotient of the constant term of $D$ by its leading coefficient implies that $\lambda^2$ is rational. If $d$ is odd then the expression of the leading coefficient of $P$ leads to the relation $\lambda^d\in\Bbb Q$ and thus $\lambda,\nu$ and $\mu$ are rational. We then have the identity
\[P+Q\sqrt{D}=\left(\lambda X+\nu+\mu\sqrt{D}\right)^d.\]
Evaluating at $0$, we obtain the relation
\[a+b\sqrt{n}=\left(\nu+\mu\sqrt{n}\right)^d,\]
In particular, the element $\nu+\mu\sqrt{n}\in K$ is an algebraic integer; it therefore belongs to $U$. The group $U/V(a,b)$ being cyclic, we deduce that $d$ divides $(U:V(a,b))$. Suppose now that $d=2m$ is even. Setting
\[\begin{cases}
P_1=4(\lambda X+\nu)^2-1,\\
Q_1=2\lambda\mu X+2\nu\mu,
\end{cases}\]
we have the identity
\[P+Q\sqrt{D}=\left(P_1+Q_1\sqrt{D}\right)^m.\]
Remark that the above discussion implies that $P_1$ and $Q_1$ belong to $\Bbb Q[X]$. Proceeding as before, we deduce that $m$ divides $(U:V(a,b))$.

Let now $d$ be an integer dividing $2(U:V(a,b))$. As before, we start with the case where $d$ is odd, so that it divides $(U:V(a,b))$. From the definition of $V(a,b)$, there exists an element $u+v\sqrt{n}$ of $U$ such that
\[a+b\sqrt{n}=\left(u+v\sqrt{n}\right)^d.\]
Remark that $2u$ and $2nv$ are integers (actually, setting $n=r^2s$, with $s$ square-free, it follows that $2rv$ is an integer). In particular, the elements $u$ and $v$ are rational. In this case, the polynomials
\[\begin{cases}
P=T_d(X+u),\\
Q=vU_{d-1}(X+u),\\
D=v^{-2}X^2+2v^{-2}uX+n,
\end{cases}\]
belong to $\Bbb Q[X]$ and define a parametric solution of degree $d$ associated to $(a,b)$. We then apply lemma~\ref{lemm1}. Suppose now that $d=2m$ is even, so that $m$ divides $(U:V(a,b))$. As before, there exists $u+v\sqrt{n}\in U$ such that
\[a+b\sqrt{n}=\left(u+v\sqrt{n}\right)^m.\]
Remark that the polynomials $P_0,Q_0$ and $D_0$ defined in roposition~\ref{prop2} also work for rational solutions of the Pell equation
\[x^2-ny^2=1.\]
In the present case, the polynomials
\[\begin{cases}
P_1=v^4(u+1)X^2+2v^2(u+1)X+u,\\
Q_1=v^3X+v,\\
D=(u+1)^2v^2X^2+2(u+1)^2X+n,
\end{cases}\]
belong to $\Bbb Q[X]$ and define a parametric solution of degree $2$ associated to $(u,v)$. Setting
\[P+Q\sqrt{D}=\left(P_1+Q_1\sqrt{D}\right)^m,\]
the polynomials $P,Q,D\in\Bbb Q[X]$ are a parametric solution associated to $(a,b)$ and we apply once again lemma~\ref{lemm1}.
\end{proof}

This last result implies that the degree of a parametric solution associated to a fixed solution $(a,b)$ is bounded. It is then natural investigate the existence of an uniform bound. From the above proof, it easily follows that the answer is negative if we don't restrict to fundamental solutions. The following result shows that without any assumtion on the integer $n$, such an uniform bound does not exist even when restricting to fundamental solutions.

\begin{corollary}\label{cor1} For any positive integer $d$, there exist an integer $n$ which is not a square, a fundamental solution of the Pell equation
\[x^2-ny^2=1\]
and a parametric solution of degree $d$ associated to it.
\end{corollary}

\begin{proof} Setting
\[\left(2+\sqrt{3}\right)^d=a+b\sqrt{3}\]
and $n=3b^2$, the couple $(a,1)$ is a fundamental solution of the Pell equation
\[x^2-ny^2=1.\]
By construction, the integer $(U:V(a,b))$ is a multiple of $d$ and we can apply theorem~\ref{th2}.
\end{proof}

We close the paper by showing that when restricting to square-free integers $n$, the degree of a parametric solution is bounded by $6$. This will follow from the elementary result below.

\begin{lemma}\label{lemm2} If the integer $n$ is square-free then the group $U/V$ is cyclic of order dividing $3$.
\end{lemma}

\begin{proof} We have to prove that for any $x\in U$, we have $x^3\in\Bbb Z[\sqrt{n}]$. Since for $n$ congruent to $2$ or $3$ modulo $4$ we have the identity $\mathcal O_K=\Bbb Z[\sqrt{n}]$, we can assume $n\equiv1\pmod4$, so that an element $x\in\mathcal O_K$ can be uniquely written as $x=\frac12(a+b\sqrt{n})$, where $a,b\in\Bbb Z$ have the same parity. Such an element belongs to $U$ if and only if $a^2-nb^2=4$, in which case we find the identities
\[8x^3=a(a^2+3nb^2)+b(3a^2+nb^2)\sqrt{n}=4a(a^2-3)+4b(a^2-1)\sqrt{n}.\]
The integers $a$ and $b$ having the same parity, it follows that $a(a^2-1)$ and $b(a^2-3)$ are even, so that $x^3$ belongs to $\Bbb Z[\sqrt{n}]$.
\end{proof}

\begin{corollary}\label{cor2} Let $n$ be a positive, square-free integer. The degree of a parametric solution associated to a fundamental solution of the Pell equation
\[x^2-ny^2=1\]
is bounded by $6$. More precisely, for $n$ congruent to $2$ or $3$ modulo $4$ the degree is $1$ or $2$ while for $n$ congruent to $1$ modulo $4$ it is equal to $1,2,3$ or $6$.
\end{corollary}

\begin{proof} It is an immediate consequence of theorem~\ref{th2} combined with lemma~\ref{lemm2}.
\end{proof}

\end{document}